\documentclass[12pt,fleqn,leqno]{amsart}
\usepackage{amsfonts,amssymb}
\usepackage{mathrsfs}
\usepackage{enumitem}

\usepackage[svgnames]{xcolor}
\usepackage[colorlinks,linkcolor=blue,citecolor=Green]{hyperref}
\usepackage{titlesec}
\usepackage{tikz-cd}
\usetikzlibrary{decorations.pathmorphing}

\titleformat{\section}[block]
 {\bfseries}
 {\thesection.}
 {\fontdimen2\font}
 {}

\setlist{noitemsep}
\setenumerate{labelindent=\parindent,label=\upshape{(\alph*)}}

\newtheorem{theorem}{Theorem}[section]
\newtheorem{corollary}[theorem]{Corollary}
\newtheorem{proposition}[theorem]{Proposition}

\theoremstyle{definition}
\newtheorem{remark}[theorem]{Remark}
\newtheorem{example}[theorem]{Example}

\DeclareMathOperator{\N}{\mathbb{N}}
\DeclareMathOperator{\R}{\mathbb{R}}
\DeclareMathOperator{\Q}{\mathbb{Q}}
\DeclareMathOperator{\uhr}{\upharpoonright}
\DeclareMathOperator\sto{\leadsto}
\DeclareMathOperator\coz{coz}
\DeclareMathOperator\supp{supp}
\DeclareMathOperator{\lip}{Lip}

\renewcommand{\emptyset}{\varnothing}
\numberwithin{equation}{section}

\begin{document}

\author{Valentin Gutev}

\address{Institute of Mathematics and Informatics, Bulgarian Academy
  of Sciences, Acad. G. Bonchev Street, Block 8, 1113 Sofia, Bulgaria}

\email{\href{mailto:gutev@math.bas.bg}{gutev@math.bas.bg}}



\title{On real-valued functions of Lipschitz type}

\begin{abstract}
  The classical McShane-Whitney extension theorem for Lipschitz
  functions is refined by showing that for a closed subset of the
  domain, it remains valid for any interval of the real line. This
  result is also extended to the setting of locally (pointwise)
  Lipschitz functions. In contrast to Lipschitz and pointwise
  Lipschitz extensions, the construction of locally Lipschitz
  extensions is based on Lipschitz partitions of unity of countable
  open covers of the domain. Such partitions of unity are a special
  case of a more general result obtained by Zden\v{e}k Frol\'{\i}k. To
  avoid the use of Stone's theorem (paracompactness of metrizable
  spaces), it is given a simple direct proof of this special case of
  Frol\'{\i}k's result. As an application, it is shown that the
  locally Lipschitz functions are precisely the locally finite sums of
  sequences of Lipschitz functions. Also, it is obtained a natural
  locally Lipschitz version of one of Michael's selection theorems.
\end{abstract}

\date{\today}
\maketitle

\section{Introduction}

The Lipschitz conditions measure how fast a function can change, hence
they are naturally related to differentiability. In 1919, Hans
Rademacher \cite{MR1511935} showed that each \emph{locally Lipschitz}
function ${f:U\to \R}$ defined on an open set $U\subset \R^n$ is
differentiable almost everywhere in the sense of the Lebesgue
measure. In 1923, Vyacheslav Stepanov (spelled also Wiatscheslaw
Stepanoff) \cite{MR1512177,zbMATH02591494} extended Rademacher's
result to all \emph{locally pointwise Lipschitz} functions
$f:U\to \R$. In the literature, Stepanov's theorem is traditionally
proved by applying Rademacher's theorem to Lipschitz extensions of
Lipschitz functions; a straightforward simple proof based on Lipschitz
approximations was given by Jan Mal\'y \cite{MR1687460}. These
Lipschitz conditions are discussed in detail in the next section.
Here, we will briefly discuss some aspects of the Lipschitz extension
problem.\medskip

For $K\geq 0$ and metric spaces $(X,d)$ and $(Y,\rho)$, a map
$f : X\to Y$ is \emph{$K$-Lipschitz} if $\rho(f(p),f(q))\leq K d(p,q)$
for every $p,q\in X$.  The Lipschitz extension problem deals with
conditions on the metric spaces $(X,d)$ and $(Y,\rho)$ such that every
$K$-Lipschitz map $\varphi:A\to Y$ on a subset $A\subset X$ has a
Lipschitz extension to the whole of $X$ with a minimal loss in the
Lipschitz constant $K$. It is evident that this extension problem
makes sense when the subset $A$ is nonempty, so this will be assumed
in all our considerations without being explicitly stated.\medskip

Some of the first solutions of the Lipschitz extension problem were
complementary to the Tietze extension theorem and dated back to the
works of Hahn and Banach in the 1920s, and McShane, Whitney and
Kirszbraun in the 1930s. The classical Hahn-Banach theorem is based on
a successive point-by-point procedure of extending bounded linear
functionals. In the setting of a general metric domain, the conditions
are less restrictive and the extension is only required to be
Lipschitz with the same Lipschitz constant. In this case, the
successive procedure can be replaced by a much simpler one. The
following theorem was obtained by McShane \cite[Theorem 1]{MR1562984},
see also Whitney \cite[the footnote on p.\ 63]{MR1501735}, and is
commonly called the \emph{McShane-Whitney extension theorem}.

\begin{theorem}
  \label{theorem-Lipschitz-ext:12}
Let $(X,d)$ be a metric space and $A\subset X$. Then each
  $K$-Lipschitz function $\varphi:A\to \R$ can be extended to a
  $K$-Lipschitz function $f:X\to \R$.
\end{theorem}

The extension in Theorem \ref{theorem-Lipschitz-ext:12} was defined by
an explicit formula involving the function $\varphi$, the Lipschitz
constant $K$ and the metric $d$. It is a special case of the following
more general construction summarised in \cite[Theorem
4.1.1]{zbMATH07045619} and \cite[Theorem 2.3]{Gutev2020}.

\begin{theorem}
  \label{theorem-Lipschitz-ext:13}
  Let $(X,d)$ be a metric space, $A\subset X$ and $\varphi:A\to \R$ be
  a $K$-Lip\-schitz function. Define functions
  $\Phi_-:X\to (-\infty,+\infty]$ and
  ${\Phi_+:X\to [-\infty,+\infty)}$ by
  \begin{equation}
    \label{eq:Lipschitz-ext:35}
    \begin{cases}
      \Phi_-(p)=\sup_{x\in A}\left[\varphi(x)-K d(x,p)\right] &\text{and}\\
      \Phi_+(p)=\inf_{x\in A}\left[\varphi(x)+K d(x,p)\right], & p\in X.
    \end{cases}
  \end{equation}
  Then $\Phi_-\leq \Phi_+$ and\/ $\Phi_-,\Phi_+:X\to \R$ are
  $K$-Lipschitz extensions of $\varphi$.
\end{theorem}

The Lipschitz extension $\Phi_-$ in Theorem
\ref{theorem-Lipschitz-ext:13} represents McShane and Whitney's
approach to show Theorem \ref{theorem-Lipschitz-ext:12}. The other
extension $\Phi_+$ was used by Czipszer and Geh\'er \cite{MR71493} for
an alternative proof of the same theorem. Furthermore, the
construction in \eqref{eq:Lipschitz-ext:35} is nearly the same as the
one used in the proof of the classical Hahn-Banach theorem. This
similarity has been noted by some authors, for instance in Czipszer
and Geh\'er \cite{MR71493}. In contrast, as pointed out in
\cite{Gutev2020}, the relationship with Hausdorff's work
\cite{hausdorff:19} remained somehow unnoticed.  Namely, the Lipschitz
extensions in \eqref{eq:Lipschitz-ext:35} are also virtually the same
as the Pasch-Hausdorff construction of Lipschitz maps in
\cite{hausdorff:19}.\medskip

The functions $\Phi_-$ and $\Phi_+$ are also naturally related to each
other. This was implicitly used by Czipszer and Geh\'er in the proof
of their extension result \cite[Theorem II]{MR71493} for pointwise
Lipschitz functions (see Theorem
\ref{theorem-Loc-Lipschitz-v17:1}). Namely, we can consider these
functions as a pair of $K$-Lipschitz extension operators ``$\Lambda$''
for the $K$-Lipschitz functions on the subset $A\subset X$. According
to this interpretation, it is more natural to write that
$\Phi_-=\Lambda_-[\varphi]$ and $\Phi_+=\Lambda_+[\varphi]$ for the
given $K$-Lipschitz function $\varphi:A\to \R$. In these more general
terms, \eqref{eq:Lipschitz-ext:35} obviously implies that
\begin{equation}
  \label{eq:Lipschitz-ext:12}
  \Lambda_{-}[\varphi]=-\Lambda_+[-\varphi]\quad \text{for each
    $K$-Lipschitz function $\varphi:A\to \R$.} 
\end{equation}     

The advantage of using both functions $\Phi_-$ and $\Phi_+$, as done
in Theorem \ref{theorem-Lipschitz-ext:13}, lies in the simplification
of the proof. Briefly, for $x,y\in A$ and $p\in X$, it follows from
the inequality $\varphi(x)-\varphi(y)\leq Kd(x,y)\leq Kd(x,p)+Kd(y,p)$
that $\Phi_-(p)\leq \Phi_+(p)$. In particular,
$\varphi(p)\leq \Phi_-(p)\leq \Phi_+(p)\leq \varphi(p)$ whenever
$p\in A$. The verification that $\Phi_-$ and $\Phi_+$ are
$K$-Lipschitz is also simple. Indeed, for ${p,q\in X}$, adding
$\varphi(x)$ to both sides of $K d(x,p)\leq K d(x,q)+ K d(q,p)$,
$x\in A$, it follows at once that $\Phi_+(p)\leq \Phi_+(q)+Kd(p,q)$.
Therefore, $\Phi_+$ is $K$-Lipschitz and by
\eqref{eq:Lipschitz-ext:12}, $\Phi_-$ is $K$-Lipschitz as
well.\medskip

It was remarked in \cite{MR71493} and shown in \cite[Theorem
4.1.1]{zbMATH07045619} (see also \cite[Theorem 1]{Aronsson1967} and
\cite[Remark 2.4]{Gutev2020}) that any $K$-Lipschitz extension
$f:X\to \R$ of a $K$-Lipschitz function $\varphi:A\to \R$ should
satisfy the inequalities $\Phi_-\leq f\leq \Phi_+$. In fact, this
property is also implicit in the proof of the classical Hahn-Banach
theorem. Moreover, it is in good accord with the fact that each convex
combination of the functions $\Phi_-$ and $\Phi_+$ gives a
$K$-Lipschitz extension of $\varphi:A\to \R$. Based on this, we will
show that McShane-Whitney's extension theorem remains valid if the
real line $\R$ is replaced by any interval $\Delta\subset \R$ and
$A\subset X$ is assumed to be closed. The following theorem will be
proved in Section \ref{sec:extens-lipsch-funct}.

\begin{theorem}
  \label{theorem-Lipschitz-ext:20}
  Let $(X,d)$ be a metric space, $A\subset X$ be a closed set and
  $\Delta\subset \R$ be an interval. Then each $K$-Lipschitz function
  $\varphi:A\to \Delta$ can be extended to a $K$-Lipschitz function
  $f:X\to \Delta$.
\end{theorem}

A special case of Theorem \ref{theorem-Lipschitz-ext:20} was pointed
out in \cite{Beer2023}. Namely, as remarked after the proof of
\cite[Theorem 3.5]{Beer2023}, the $K$-Lipschitz extension $\Phi_+$ is
strictly positive provided so is the $K$-Lipschitz function
$\varphi:A\to \R$ and the subset $A\subset X$ is closed. Finally, let
us also remark that Theorem \ref{theorem-Lipschitz-ext:20} is
complementary to the Dugundji extension theorem \cite[Theorem
4.1]{dugundji:51}, and its special case for the real line in
\cite[Theorem 3.1$'''$]{michael:56a}. \medskip

Based on the same approach, we will show in Section
\ref{sec:extens-pointw-lipsch} that Theorem
\ref{theorem-Lipschitz-ext:20} is also valid in the setting of
(locally) pointwise Lipschitz functions which improves a result
obtained previously by Czipszer and Geh\'er \cite[Theorem
II]{MR71493}, see also \cite[Theorem 3.2]{Gutev2020}. The rest of the
paper deals with locally Lipschitz functions. Unlike (pointwise)
Lipschitz functions, the extension of locally Lipschitz functions is
based on Theorem \ref{theorem-Lipschitz-ext:20} and Lipschitz
partitions of unity. It was shown by Zden\v{e}k Frol\'{\i}k in
\cite[Theorem 1]{MR814046} that each open cover of a metric space
admits a locally finite partition of unity consisting of Lipschitz
functions. However, this result is essentially equivalent to Stone's
theorem \cite[Corollary 1]{stone:48} that each metrizable space is
paracompact. In Section~\ref{sec:lipsch-locally-lipsc}, we will show
that in \cite[Lemma]{MR814046}\,---\,the special case of countable
covers in the aforementioned Frol\'{\i}k's result, the use of Stone's
theorem can be avoided (see Theorem
\ref{theorem-Loc-Lipschitz-Ext:2}). Next, in Section
\ref{sec:lipsch-locally-lipsc-1}, we will show that each locally
Lipschitz function is a locally finite countable sum of Lipschitz
functions (Theorem \ref{theorem-Loc-Lipschitz-Ext:3}). In the same
section, we will obtain another interesting result that a map between
metric spaces is locally Lipschitz precisely when its ``Lipschitz
constant'' varies continuously over the square of the domain
(Theorem~\ref{theorem-LE-Revisited-v17:1}). Finally, in
Section~\ref{sec:extens-locally-lipsc}, we will show that Theorem
\ref{theorem-Lipschitz-ext:20} is also valid for locally Lipschitz
functions (Theorem \ref{theorem-Loc-Lipschitz-v5:1}). Moreover, based
on the same approach, we will obtain a natural locally Lipschitz
version of Michael's selection theorem \cite[Theorem
3.1$'''$]{michael:56a}, see Theorems \ref{theorem-LE-Revisited-v14:1}
and \ref{theorem-Lipschitz-vgg-r1:2}. Such selections are very
efficient to show that each continuous function on a metric space is
the uniform limit of a strictly decreasing (equivalently, strictly
increasing) sequence of locally Lipschitz functions, see Corollary
\ref{corollary-Lipschitz-vgg-r1:1}.

\section{Lipschitz and Locally Lipschitz Conditions}

Let $(X,d)$ and $(Y,\rho)$ be metric spaces. To each map $f:X\to Y$ we
may associate the (possibly infinite) number $\lip(f)$ defined by
\begin{equation}
  \label{eq:uniform-ext:11}
  \lip(f)=\sup_{p\neq q}\frac{\rho(f(p),f(q))}{d(p,q)}.
\end{equation}
This number can be regarded as the general \emph{Lipschitz constant}
of $f$. Namely, $\lip(f) =+\infty$ indicates that $f$ is not
Lipschitz, and $\lip(f)=0$ indicates that $f$ is constant. Thus, the
essential case of Lipschitz maps is when $0<\lip(f)<+\infty$. \medskip

The \emph{Lipschitz constant} of a map $f:X\to Y$ at a point $p\in X$
is defined by
\begin{equation}
  \label{eq:Lipschitz-const:2}
  \lip(f,p)=\sup_{x\neq p}\frac{\rho(f(x),f(p))}{d(x,p)}.
\end{equation}
We say that $f:X\to Y$ is \emph{pointwise Lipschitz} if
$\lip(f,p)<+\infty$ for every $p\in X$. In other words, $f:X\to Y$ is
pointwise Lipschitz if and only if for each $p\in X$ there exists
$L_p\geq0$ such that
\begin{equation}
  \label{eq:Lipschitz-const:3}
  \rho(f(x),f(p))\leq L_p d(x,p)\quad \text{for every $x\in X$.}
\end{equation}

For $\delta>0$, let $\mathbf{O}(p,\delta)= \{x\in X: d(x,p)<\delta\}$
be the \emph{open $\delta$-ball centred at a point $p\in X$}.
A map $f:X\to Y$ is \emph{locally Lipschitz} if for each point
$p\in X$ there exists $\delta_p>0$ such that
$f\uhr \mathbf{O}(p,\delta_p)$ is Lipschitz. Similarly, we say that
$f:X\to Y$ is \emph{locally pointwise Lipschitz} if for each $p\in X$
there exists $K_p\geq 0$ and $\delta_p>0$ such that
\begin{equation}
  \label{eq:Lipschitz-ext:30}
  \rho(f(x),f(p))\leq K_p d(x,p)\quad \text{for every $x\in
    \mathbf{O}(p,\delta_p)$.} 
\end{equation}

The locally pointwise Lipschitz condition is naturally related to the
extended function $\lip_f:X\to [0,+\infty]$ associated to $f:X\to Y$
by letting, for $p\in X$,
\begin{equation}
  \label{eq:Lipschitz-const:1}
  \lip_f(p)=\inf_{t> 0}\left[\sup_{d(x,p)<t}
    \frac{\rho(f(x),f(p))}{t}\right] =
  \varlimsup_{t\to 0^+}\left[\sup_{d(x,p)<t}
    \frac{\rho(f(x),f(p))}{t}\right].
\end{equation}
This function is well know, see e.g.\
\cite{Cheeger1999,Keith2004a,Keith2004}, and in \cite{Balogh2004} it
was called the \emph{upper scaled oscillation} of $f$. Subsequently,
in \cite{Heinonen2007}, its values $\lip_f(p)$, $p\in X$, were called
the \emph{pointwise inﬁnitesimal Lipschitz numbers} of $f$. It follows
easily from \eqref{eq:Lipschitz-const:1} that $\lip_f(p)=0$ for each
isolated point $p\in X$, and
$\lip_f(p)=\varlimsup_{x\to p}\frac{\rho(f(x),f(p))}{d(x,p)}$
otherwise. Accordingly, $f:X\to Y$ is locally pointwise Lipschitz
precisely when $\lip_f:X\to [0,+\infty)$ is a usual
function.\medskip 

Finally, let us remark that (locally) pointwise Lipschitz maps have
been used but not explicitly defined in the literature. Namely, the
term ``pointwise Lipschitz'' is used in \cite{MR2564873} and
\cite{Gutev2020} to refer to the \emph{local property} defined in
\eqref{eq:Lipschitz-ext:30}. It was recently proposed in
\cite{Aggarwal2024,Aggarwal2023} that the local property in
\eqref{eq:Lipschitz-ext:30} be called ``pointwise locally Lipschitz'',
see \cite[Remark 2.2] { Aggarwal2024}. We will show below that our
terminology better reflects the relationship between these
Lipschitz-like conditions. In different terms, the following simple
observation was obtained in \cite[Proposition 3.1]{Gutev2020}.

\begin{proposition}
  \label{proposition-Loc-Lipschitz-v3:2}
  If $(X,d)$ and $(Y,\rho)$ are metric spaces, then each bounded
  locally pointwise Lipschitz map $f:X\to Y$ is pointwise Lipschitz.
\end{proposition}

\begin{proof}
  Let $p\in X$ and $\rho(f(x),f(p))\leq K$ for some $K\geq 0$ and all
  $x\in X$. Also, let $K_p\geq 0$ and $\delta_p>0$ be as in
  \eqref{eq:Lipschitz-ext:30}. Then
  $L_p= \max\left\{K_p,\frac K{\delta_p}\right\}$ is as in
  \eqref{eq:Lipschitz-const:3}.
\end{proof}

Each locally pointwise Lipschitz map is continuous, hence it is also
locally bounded. By Proposition \ref{proposition-Loc-Lipschitz-v3:2},
this implies the following immediate consequence.

\begin{corollary}
  \label{corollary-Lipschitz-const:1}
  If $(X,d)$ and $(Y,\rho)$ are metric spaces, then a map $f:X\to Y$
  is locally pointwise Lipschitz if and only if for each $p\in X$
  there exists $\delta_p>0$ such that $f\uhr \mathbf{O}(p,\delta_p)$
  is pointwise Lipschitz.
\end{corollary}

In contrast to Proposition \ref{proposition-Loc-Lipschitz-v3:2},
bounded locally Lipschitz maps are not necessarily Lipschitz. Here is
a simple example.

\begin{example}
  \label{example-Lipschitz-const:1}
  The function $f(t)=\sin \frac 1t$, $t>0$, is clearly bounded and
  locally Lipschitz but is not Lipschitz (it is even not uniformly
  continuous). This easily follows by taking $s_n=\frac2{(4n+1)\pi}$
  and $t_n=\frac2{(4n+3)\pi}$, $n\in \N$. Then
  $|s_n-t_n|\xrightarrow[n\to \infty]{} 0$ and $|f(s_n)-f(t_n)|=2$,
  for every $n\in \N$.\qed
\end{example}

However, each locally Lipschitz map on a compact space is
Lipschitz. The following more general fact was obtained by Scanlon in
\cite[Theorem 2.1]{MR253053}.

\begin{theorem}
  \label{theorem-Lipschitz-vgg-r1:1}
  For metric spaces $(X,d)$ and $(Y,\rho)$, a map $f:X\to Y$ is
  locally Lipschitz if and only if its restriction on each compact
  subset of $X$ is Lipschitz.
\end{theorem}                 
  
Each (pointwise) Lipschitz map is locally (pointwise) Lipschitz, and
each locally Lipschitz map is locally pointwise Lipschitz. The
converse is not true, the function
$\frac1t:(0,+\infty)\to (0,+\infty)$ is a simple example of a locally
Lipschitz function which is neither Lipschitz nor pointwise
Lipschitz. Finally, let us also remark that there are pointwise
Lipschitz maps which are not locally Lipschitz. The following example
was given in \cite[Example 2.7]{MR2564873}.

\begin{example}
  \label{example-Lipschitz-ext:5}
  Let
  $X=\left\{\left(t^3,t^2\right): -1\leq t\leq 1\right\}\subset \R^2$
  and for convenience, set $\mathbf{u}_t=\left(t^3,t^2\right)$,
  $t\in [-1,1]$. Next, define a function $f:X\to \R$ by
  $f(\mathbf{u}_t)=t^2$ if $t\geq 0$, and $f(\mathbf{u}_t)=-t^2$
  otherwise.  If $s,t\in [-1,1]$ with $s\cdot t\geq 0$ and $s\neq t$,
  then
  \[
    \frac{\left|f(\mathbf{u}_s)-f(\mathbf{u}_t)\right|}
    {\|\mathbf{u}_s-\mathbf{u}_t\|}= \frac{\left|s^2-t^2\right|}
    {\sqrt{\left(s^3-t^3\right)^2 +\left(s^2-t^2\right)^2}}\leq
    \frac{\left|s^2-t^2\right|}{\left|s^2-t^2\right|}=1.
  \]
  Accordingly, $f$ is locally pointwise Lipschitz. Hence, by
  Proposition \ref{proposition-Loc-Lipschitz-v3:2}, it is also
  pointwise Lipschitz being bounded. However, $f$ is not locally
  Lipschitz because for $t\in [-1,1]$ with $t>0$,
  $ \frac{\left|f(\mathbf{u}_{t})-f(\mathbf{u}_{-t})\right|}
  {\|\mathbf{u}_{t}-\mathbf{u}_{-t}\|}= \frac{\left|t^2+t^2\right|}
  {\sqrt{\left(t^3+t^3\right)^2 +\left(t^2-t^2\right)^2}} =
  \frac1t\xrightarrow[t\to 0]{}\infty$. \qed
\end{example}

\section{Extension of Lipschitz Functions}
\label{sec:extens-lipsch-funct}

In this section we will prove Theorem
\ref{theorem-Lipschitz-ext:20}. Obviously, this theorem is trivial
when the function $\varphi$ is constant (i.e.\ $0$-Lipschitz). So our
considerations below are for the nontrivial case of $K$-Lipschitz
functions with $K>0$.\medskip

As mentioned in the Introduction, a special case of Theorem
\ref{theorem-Lipschitz-ext:20} was pointed out in
\cite{Beer2023}. Namely, for a closed subset $A\subset X$ of a metric
space $(X,d)$ and a $K$-Lipschitz function $\varphi:A\to \R$, let
$\Phi_-$ and $\Phi_+$ be the $K$-Lipschitz extensions of $\varphi$
defined as in \eqref{eq:Lipschitz-ext:35}. Then as stated after the
proof of \cite[Theorem 3.5]{Beer2023}, the extension $\Phi_+$ is
strictly positive provided so is $\varphi:A\to \R$.  This is evident
in a more general situation because
\begin{equation}
  \label{eq:Lipschitz-ext:18}
  \begin{cases}
    \Phi_-(p)\leq \sup_{x\in A}\varphi(x)-Kd(p,A)&
    \text{and} \\
    \Phi_+(p)\geq \ \inf_{x\in A}\varphi(x)+Kd(p,A),& p\in X.
  \end{cases}
\end{equation}
Accordingly, the special case of Theorem
\ref{theorem-Lipschitz-ext:20} when the interval $\Delta\subset \R$ is
unbounded now follows with ease.

\begin{proposition}
  \label{proposition-Lipschitz-ext:1}
  Let $(X,d)$ be a metric space, $A\subset X$ be closed and
  $\Delta\subset \R$ be an unbounded interval. Then each $K$-Lipschitz
  function $\varphi:A\to \Delta$ can be extended to a $K$-Lipschitz
  function $f:X\to \Delta$.
\end{proposition}

\begin{proof}
  The case of $\Delta=\R$ is covered by Theorem
  \ref{theorem-Lipschitz-ext:12}. If $\Delta\neq \R$, then either
  $(-\infty,r)\subset \Delta\subset (-\infty,r]$ or
  $(r,+\infty)\subset \Delta\subset [r,+\infty)$ for some $r\in
  \R$. In this case, the extension property follows from Theorem
  \ref{theorem-Lipschitz-ext:13} because by
  \eqref{eq:Lipschitz-ext:18}, either $\Phi_-:X\to \Delta$ or
  $\Phi_+:X\to \Delta$.
\end{proof}

Let $\Delta\subset \R$ be a bounded interval, i.e.\
$(a,b)\subset \Delta\subset [a,b]$ for some $a,b\in \R$ with $a<b$,
and $\varphi:A\to \Delta$ be a $K$-Lipschitz function from some closed
subset $A\subset X$ of a metric space $(X,d)$. Then by Theorem
\ref{theorem-Lipschitz-ext:12}, $\varphi$ can be extended to a
$K$-Lipschitz function $f:X\to \R$. Hence, $\varphi$ can also be
extended to a $K$-Lipschitz function $g:X\to [a,b]$ by taking
$g(x)=f(x)$ if $f(x)\in [a,b]$, $g(x)=a$ if $f(x)<a$ and $g(x)=b$ if
$f(x)>b$, see \cite[Remark 4.1.2]{zbMATH07045619}. However, some of
the values of $g$ may leave the interval $\Delta$. In contrast,
applying the same construction to each of the $K$-Lipschitz extensions
defined in \eqref{eq:Lipschitz-ext:35}, the arithmetic mean of the
resulting functions is a $K$-Lipschitz extension of $\varphi$ in the
interval $\Delta$. This covers the remaining case in Theorem
\ref{theorem-Lipschitz-ext:20}, in fact the following more general
result holds.

\begin{proposition}
  \label{proposition-Lipschitz-ext:11}
  Let $(X,d)$ be a metric space, $A\subset X$ be a closed set and
  $a,b\in \R$ with $a<b$. Then each $K$-Lipschitz function
  $\varphi:A\to [a,b]$ can be extended to a $K$-Lipschitz function
  $f:X\to [a,b]$ such that $f\left(X\setminus A\right)\subset (a,b)$.
\end{proposition}

\begin{proof}
  We argue as in \cite[Theorem II]{MR71493}. Briefly, let
  $\varphi:A\to [a,b]$ be a $K$-Lipschitz function and
  $\Phi_-,\Phi_+:X\to \R$ be the $K$-Lipschitz extensions of $\varphi$
  defined as in \eqref{eq:Lipschitz-ext:35}, see Theorem
  \ref{theorem-Lipschitz-ext:13}.  Then, according to
  \eqref{eq:Lipschitz-ext:18}, we can define two other $K$-Lipschitz
  extensions $\Psi_-,\Psi_+:X\to [a,b]$ by
  \[
    \Psi_-(p)=\max\{\Phi_-(p),a\}\ \ \text{and}\ \ 
    \Psi_+(p)=\min\{\Phi_+(p),b\},\quad \text{$p\in X$.}   
  \]
  The function $f=\frac{\Psi_-+\Psi_+}2:X\to [a,b]$ is now as
  required. Indeed, if $p\in X\setminus A$, then $d(p,A)>0$ and by
  \eqref{eq:Lipschitz-ext:18},
  $f(p)=\frac{\Psi_-(p)+\Psi_+(p)}2\in (a,b)$.
\end{proof}

Finally, let us note that in Theorem \ref{theorem-Lipschitz-ext:20},
the requirement that the subset $A\subset X$ be closed cannot be
dropped. For instance, take any open interval
$\Delta=(a,b)\subset \R$, $a<b$. Then the identity function
$\varphi:\Delta\to \Delta$ is clearly Lipschitz, but cannot be
extended even to a continuous function $f:\R\to \Delta$.

\section{Extension of Pointwise Lipschitz Functions}
\label{sec:extens-pointw-lipsch}

In 1955, Czipszer and Geh\'er \cite{MR71493} applied the construction
in Theorem \ref{theorem-Lipschitz-ext:13} to show that the
McShane-Whitney extension theorem is also valid for (locally)
pointwise Lipschitz functions. Namely, in \cite[Theorem II]{MR71493}
they obtained the special case of open symmetric intervals
$\Delta\subset \R$ of the following more general extension result,
compare with Theorem \ref{theorem-Lipschitz-ext:20}.

\begin{theorem}
  \label{theorem-Lipschitz-ext:11}
  Let $(X,d)$ be a metric space, $A\subset X$ be a closed set and
  $\Delta\subset \R$ be an interval. Then each locally pointwise
  Lipschitz function $\varphi:A\to \Delta$ can be extended to a
  locally pointwise Lipschitz function $f:X\to \Delta$.
\end{theorem}

The proof of Theorem \ref{theorem-Lipschitz-ext:11} is similar to that
of \cite[Theorem II]{MR71493}, see also \cite[Theorem 3.2 and Remark
3.5]{Gutev2020}. In fact, the most difficult part of this proof is the
case when $\Delta\subset\R$ is a bounded interval and in particular
the verification that the extension is locally (pointwise) Lipschitz
at $X\setminus A$. For a bounded interval $\Delta\subset\R$, according
to Proposition \ref{proposition-Loc-Lipschitz-v3:2}, each locally
pointwise Lipschitz function $\varphi:A\to \Delta$ is pointwise
Lipschitz. Thus, the essential part of the proof of Theorem
\ref{theorem-Lipschitz-ext:11} is actually based on the following
extension result for pointwise Lipschitz functions, compare with
Proposition \ref{proposition-Lipschitz-ext:11}.

\begin{theorem}
  \label{theorem-Loc-Lipschitz-v17:1}
  Let $(X,d)$ be a metric space, $A\subset X$ be a closed set and
  $a,b\in \R$ with $a<b$. Then each pointwise Lipschitz function
  $\varphi:A\to [a,b]$ can be extended to a pointwise Lipschitz
  function $f:X\to [a,b]$ such that $f(X\setminus A)\subset (a,b)$.
\end{theorem}

\begin{proof}
  Let $\varphi:A\to [a,b]$ be a pointwise Lipschitz function and
  $L_x\geq 1$, $x\in A$, be as in \eqref{eq:Lipschitz-const:3}. Then
  as in \eqref{eq:Lipschitz-ext:35}, we can associate the functions
  $\Phi_-:X\to (-\infty,b]$ and $\Phi_+:X\to [a,+\infty)$ defined by
  \begin{equation}
    \label{eq:Loc-Lipschitz-v4:2}
    \begin{cases}
      \Phi_-(p)=\sup_{x\in A}\left[\varphi(x)-L_x d(x,p)\right]
      &\text{and}\\
      \Phi_+(p)=\ \inf_{x\in A}\left[\varphi(x)+L_x d(x,p)\right],
      &\text{for every $p\in X$.}
    \end{cases}
  \end{equation}

  In this more general setting, just as in Theorem
  \ref{theorem-Lipschitz-ext:13}, we have that
  \begin{equation}
    \label{eq:Loc-Lipschitz-v18:1}
    \Phi_-\leq \Phi_+\quad\text{and}\quad \Phi_-\uhr
    A=\varphi=\Phi_+\uhr A. 
  \end{equation}
  Indeed, for $x,y\in A$, let $L_{xy}=\min\{L_x,L_y\}$. Then for every
  $p\in X$, it follows from \eqref{eq:Lipschitz-const:3} that
  $\varphi(x)-\varphi(y)\leq L_{xy} d(x,y)\leq L_x d(x,p) + L_y
  d(p,y)$.  In other words,
  $\varphi(x)-L_x d(x,p)\leq \varphi(y)+L_y d(y,p)$ which is clearly
  equivalent to $\Phi_-(p)\leq \Phi_+(p)$. Similarly, taking $p\in A$,
  it follows from \eqref{eq:Loc-Lipschitz-v4:2} that
  $\varphi(p)\leq \Phi_-(p)\leq \Phi_+(p)\leq \varphi(p)$, therefore
  $\Phi_-\uhr A=\varphi= \Phi_+\uhr A$. \medskip

  Finally, as might be expected, $\Phi_-$ and $\Phi_+$ have the same
  property as bounded Lipschitz functions, see
  \eqref{eq:Lipschitz-ext:18}. Namely, since $L_x\geq 1$, $x\in A$,
  \eqref{eq:Loc-Lipschitz-v4:2} implies that
  \begin{equation}
    \label{eq:Lip-Like-Ext:2}    
    \begin{cases}
      \Phi_-(p)\leq b-d(p,A) &\text{and} \\
      \Phi_+(p)\geq a+d(p,A), &\text{for every $p\in X$.}
    \end{cases}
  \end{equation}

  The essential part of the proof now is to show that the extensions
  $\Phi_-$ and $\Phi_+$ are locally pointwise Lipschitz. Since
  $\lip(-\varphi,x)=\lip(\varphi,x)$ for every $x\in A$, see
  \eqref{eq:Lipschitz-const:2}, just like in the proof of Theorem
  \ref{theorem-Lipschitz-ext:13}, it suffices to show that one of
  these functions, say $\Phi_+$, is locally pointwise Lipschitz. So,
  take a point $p\in X$. \medskip

  If $p\in A$ and $q\in X$, then by \eqref{eq:Loc-Lipschitz-v18:1},
  $\Phi_-(q)\leq \Phi_+(q)$ and $\Phi_+(p)=\varphi(p)$. Hence,
  \eqref{eq:Loc-Lipschitz-v4:2} implies that
  $\Phi_+(p)-L_p d(p,q)\leq \Phi_-(q)\leq \Phi_+(q)\leq \Phi_+(p)+ L_p
  d(p,q)$.  Thus, $|\Phi_+(p)-\Phi_+(q)|\leq L_p d(p,q)$.\medskip

  \noindent\begin{minipage}{0.5\linewidth}
    If $p\notin A$, then $\Phi_+$ is bounded in the open $\delta$-ball
    $\mathbf{O}(p,\delta)$, where $\delta=\frac{d(p,A)}2>0$. Indeed,
    fix a point $\varkappa\in A\cap \mathbf{O}(p,3\delta)$ and set
    ${K=L_{\varkappa}}$. Then for each $z\in \mathbf{O}(p,\delta)$, it
    follows from \eqref{eq:Loc-Lipschitz-v4:2} that
  \end{minipage}
  \begin{minipage}[c][60pt][t]{0.45\linewidth}
    \hspace{25pt}
    \begin{tikzpicture}[scale=1, every node/.style={scale=0.9}]

      \filldraw[black!10!white,line width=1pt] (3,1.5) .. controls
      (1,0) and (2,-1) .. (5.5,1.5);

      \draw[black!50!white,line width=1pt] (3,1.5) .. controls (1,0)
      and (2,-1) .. (5.5,1.5) (3.5,1) node[anchor=center] {$A$};

      \draw [darkgray, line width=.4pt, dashed] (1.4,.8) circle
      [radius=20pt];

      \filldraw[black!10!white] (1.4,.8) circle [radius=10pt];

      \draw [darkgray, line width=.4pt] (1.4,.8) circle [radius=32pt]
      (1.4,.8) circle [radius=10pt];

      \filldraw [darkgray] (1.4,.8) circle [radius=.6pt]
      node[anchor=north] {\scriptsize $p$} (1.99,1.55)
      node[anchor=north,rotate=30] {\scriptsize $3\delta$} (1.85,.57)
      node[anchor=north,rotate=-32] {\scriptsize $2\delta$};

      \draw [darkgray,line width=.4pt, ->] (1.4,.8) -- (2.4,1.33);

      \draw [darkgray,line width=.4pt, ->] (1.4,.8) -- (2,.46);

      \draw [darkgray,line width=.5pt, ->] (1.4,.8) -- (1.4,1.14);

      \draw (1.38,.99) node[anchor=west,,scale=0.8] {\tiny $\delta$};

      \filldraw [black] (2.3,.5) circle [radius=.6pt] (2.18,.54)
      node[anchor=north] {\scriptsize $\varkappa$} (1.2,.75) circle
      [radius=.6pt] (1.25,.7) node[anchor=south] {\scriptsize $z$};
    \end{tikzpicture}
  \end{minipage}
  \begin{equation}
    \label{eq:Loc-Lipschitz-v19:1}
    \Phi_+(z)\leq \varphi(\varkappa)+K d(\varkappa,z) \leq b+ 4\delta
    K.
  \end{equation}
  This implies that for each $z\in \mathbf{O}(p,\delta)$ and
  $\varepsilon>0$, there exists a point $x\in A$ with
  \begin{equation}
    \label{eq:LE-Revisited-v12:1}
    \Phi_+(z)+\varepsilon> \varphi(x)+ L_xd(x,z)\quad\text{and} \quad
    L_x \leq \frac{b-a+\varepsilon}\delta+4K.
  \end{equation}
  Namely, $\Phi_+(z)+\varepsilon>\Phi_+(z)$ and by
  \eqref{eq:Loc-Lipschitz-v4:2},
  $\Phi_+(z)+\varepsilon\geq \varphi(x)+ L_x d(x,z)$ for some point
  $x\in A$. Since $\varphi(x)\geq a$ and
  $d(x,z)\geq d(z,A)\geq \delta$, it follows from
  \eqref{eq:Loc-Lipschitz-v19:1} that
  $L_x\leq \frac{\Phi_+(z)-\varphi(x)+\varepsilon}{d(x,z)} \leq
  \frac{b+4\delta K-a+\varepsilon}{\delta}\leq
  \frac{b-a+\varepsilon}\delta+4K$.\medskip

  Take now $y,z\in \mathbf{O}(p,\delta)$ and $\varepsilon>0$. Next,
  let $x\in A$ be as in \eqref{eq:LE-Revisited-v12:1} with respect to
  this $\varepsilon>0$ and the point $z$. Since
  $\Phi_+(y)\leq \varphi(x)+L_x d(x,y)$, we get that
  \begin{align*}
    \Phi_+(y)-\Phi_+(z)
    &\leq \varphi(x)+L_x d(x,y)-\varphi(x)-L_x
      d(x,z)+\varepsilon\\ 
    &\leq L_x d(y,z)+\varepsilon\leq
      \left[\frac{b-a+\varepsilon}\delta
      +4K\right] d(y,z)+ 
      \varepsilon.
  \end{align*}
  Accordingly,
  $\Phi_+(y)-\Phi_+(z)\leq \left[\frac{b-a}\delta +4K\right] d(y,z)$
  and $\Phi_+$ is Lipschitz in $\mathbf{O}(p,\delta)$. Thus, $\Phi_+$
  is locally pointwise Lipschitz. \medskip

  The proof now concludes in the same way as that of Proposition
  \ref{proposition-Lipschitz-ext:11}. Briefly, define two other
  locally pointwise Lipschitz extensions $\Psi_-,\Psi_+:X\to [a,b]$ by
  $\Psi_-(p)=\max\{\Phi_-(p),a\}$ and $\Psi_+(p)=\min\{\Phi_+(p),b\}$
  for every $p\in X$.  Then by Proposition
  \ref{proposition-Loc-Lipschitz-v3:2}, $\Psi_-$ and $\Psi_+$ are
  pointwise Lipschitz being bounded functions. Hence,
  $f=\frac{\Psi_-+\Psi_+}2:X\to [a,b]$ is also a pointwise Lipschitz
  extension of $\varphi$. This $f$ is as required. Indeed, if
  $p\in X\setminus A$, then $d(p,A)>0$ and by
  \eqref{eq:Lip-Like-Ext:2}, $\Psi_-(p)<b$ and
  $\Psi_+(p)>a$. Accordingly,
  $f(p)=\frac{\Psi_-(p)+\Psi_+(p)}2\in (a,b)$.
\end{proof}

Theorem \ref{theorem-Lipschitz-ext:11} now follows from Theorem
\ref{theorem-Loc-Lipschitz-v17:1} by using suitable locally Lipschitz
homeomorphisms between bounded and unbounded intervals.

\begin{proof}[Proof of Theorem \ref{theorem-Lipschitz-ext:11}]
  The case of a bounded interval $\Delta\subset \R$ is covered by
  Theorem \ref{theorem-Lipschitz-ext:11}. Suppose that
  $\Delta\subset \R$ is unbounded and $\varphi:A\to \Delta$ is a
  locally pointwise Lipschitz function.  If $\Delta=\R$, then as in
  \cite{MR71493}, we can use the fact that
  $\arctan:\R\to\left(-\frac\pi2,\frac\pi2\right)$ is Lipschitz, while
  its inverse ${\tan:\left(-\frac\pi2,\frac\pi2\right)\to \R}$ is
  locally Lipschitz. Thus, the composite function
  $\psi=\arctan\circ \varphi:A\to \left(-\frac\pi2,\frac\pi2\right)$
  is also locally pointwise Lipschitz.  Hence, by Proposition
  \ref{proposition-Loc-Lipschitz-v3:2} and Theorem
  \ref{theorem-Loc-Lipschitz-v17:1}, it can be extended to a pointwise
  Lipschitz function $g:X\to
  \left(-\frac\pi2,\frac\pi2\right)$. Finally, in this case,
  $f=\tan\circ g:X\to \R$ is the required extension of
  $\varphi$.\smallskip

  If $\Delta\neq \R$, then using an isometry (if necessary), the proof
  is reduced to the case when
  $(1,+\infty)\subset \Delta\subset [1,+\infty)$. Since the function
  $\frac1t:(0,+\infty)\to (0,+\infty)$ is locally Lipschitz, the
  function $\psi=\frac1\varphi:A\to (0,1]$ is locally pointwise
  Lipschitz. Hence, just like before, it can be extended to a
  pointwise Lipschitz function ${g:X\to (0,1]}$ with
  $g(X\setminus A)\subset (0,1)$. Now we can take
  $f=\frac 1g:X\to \Delta$, which is clearly a locally pointwise
  Lipschitz extension of $\varphi$.
\end{proof}

Regarding Theorems \ref{theorem-Lipschitz-ext:11} and
\ref{theorem-Loc-Lipschitz-v17:1}, let us remark that just as in Theorem
\ref{theorem-Lipschitz-ext:20}, the requirement $A\subset X$ to be
closed cannot be dropped. For instance, the function
$\frac{t}{|t|}:[-1,0)\cup(0,1]\to \{-1,1\}$ is clearly pointwise
Lipschitz, but cannot be extended continuously on $[-1,1]$.

\section{Lipschitz and Locally Lipschitz Partitions of Unity}
\label{sec:lipsch-locally-lipsc}

The \emph{cozero set}, or the \emph{set-theoretic support}, of a
function $\xi : X\to \R$ is the set
$\coz(\xi) = \{x\in X : \xi(x)\neq 0\}$. For a space $X$, the
\emph{support} of a function $\xi:X\to \R$ is the set
$\supp(\xi)=\overline{\coz(\xi)}$. A cover $\mathscr{U}$ of a space
(set) $X$ is a \emph{refinement} of another cover $\mathscr{V}$ of $X$
if each $U\in \mathscr{U}$ is contained in some $V\in
\mathscr{V}$. Finally, let us recall that a partition of unity
$\xi_\alpha:X\to [0,1]$, $\alpha\in \mathscr{A}$, on $X$ is
\emph{subordinated} to a cover $\mathscr{V}$ of $X$ if the cover
$\left\{\supp(\xi_\alpha):\alpha\in\mathscr{A}\right\}$ is a
refinement of $\mathscr{V}$, in which case we also say that
$\mathscr{V}$ admits a partition of unity. \medskip

It was shown by Zden\v{e}k Frol\'{\i}k in \cite[Theorem 1]{MR814046}
that each open cover of a metric space admits a locally finite
partition of unity consisting of Lipschitz functions. This was based
on A. H. Stone's theorem \cite[Corollary 1]{stone:48} that each
metrizable space is paracompact. Here, we will give a direct proof of
the special case of countable covers (\cite[Lemma]{MR814046}) without
using Stone's result. Namely, we will prove the following theorem.

\begin{theorem}
  \label{theorem-Loc-Lipschitz-Ext:2}
  Each countable open cover of a metric space admits a locally finite
  countable partition of unity consisting of Lipschitz functions.
\end{theorem}

The proof of Theorem \ref{theorem-Loc-Lipschitz-Ext:2} is based on two
constructions. In the first one, a family of functions
$\xi_\alpha:X\to \R$, $\alpha\in \mathscr{A}$, is called \emph{locally
  finite} if the corresponding family
$\left\{\supp(\xi_\alpha):\alpha\in\mathscr{A}\right\}$ of subsets of
$X$ is locally finite; equivalently, if the family
$\left\{\coz(\xi_\alpha):\alpha\in\mathscr{A}\right\}$ is locally
finite in $X$. The following property of the function
$\frac1t:(0,+\infty)\to (0,+\infty)$ was stated in 
\cite[Lemma]{MR814046}. 
\begin{proposition}
  \label{proposition-Loc-Lipschitz-Ext:4}
  There exists a locally finite sequence  of Lipschitz functions
  $\ell_k:(0,+\infty)\to [0,1]$, $k\in \N$, such that
  \begin{equation}
    \label{eq:Loc-Lipschitz-Ext:2}
    \sum_{k=1}^\infty \ell_k(t)=\frac 1t\qquad \text{for every $t>0$.}
  \end{equation}
\end{proposition}

\begin{proof}
  For each $k\in\N$, define a Lipschitz function
  $\ell_k:(0,+\infty)\to [0,1]$ by letting
  $\ell_1(t)=\min\left\{1,\frac1t\right\}$ and
  $\ell_{k+1}(t)=\min\left\{k+1,\frac1t\right\}-\sum_{n=1}^k\ell_n(t)$,
  for $t>0$.
    
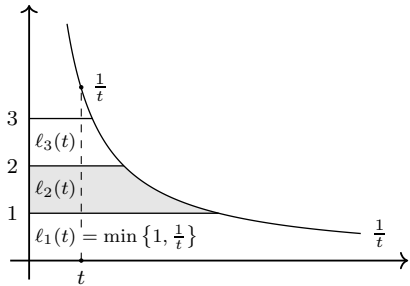
\begin{center}
      \begin{tikzpicture}[domain=-0.2:2.5,samples=500,scale=1.26]
     \filldraw[black!10!white] (0,.5) -- (2,.5) -- (1,.99) -- (0,.99);
  \draw[line width=.7pt, ->] (-0.2,0) -- (4,0);
  \draw[line width=.7pt, ->] (0,-.2) -- (0,2.7);
  \draw[fill=white, line width=.5pt,domain=0.4:3.5] plot (\x,1/\x)
  node[right] {\tiny $\frac1t$};
     \draw[line width=.5pt] (0,.5) -- (2,.5) (0,.5)
     node[left] {\tiny $1$} (-0.04,.25) node[scale=.9, right]{\tiny
       $\ell_1(t)=\min\big\{1,\frac1t\big\}$};
     \draw[line width=.5pt] (0,1) -- (1,1) (0,1)
     node[left] {\tiny $2$} (-0.04,.75) node[scale=.9, right]{\tiny
       $\ell_2(t)$};
     \draw[line width=.5pt] (0,1.5) -- (.67,1.5) (0,1.5)
     node[left] {\tiny $3$} (-0.04,1.25) node[scale=.9, right]{\tiny
       $\ell_3(t)$};
          \draw[line width=.4pt, dashed] (.55,0) -- (.55,.1)
          (.55,.35)-- (.55,1.8) (.55,0) node[below] {\tiny $t$};
       \filldraw [black]   (.55,0) circle
        [radius=.5pt] (.55,1.83) circle
        [radius=.5pt] node[right] {\tiny $\frac1t$};
      \end{tikzpicture}
    \end{center}
    It is evident that \eqref{eq:Loc-Lipschitz-Ext:2} holds because
    $t>\frac 1k$ implies that $\sum_{n=1}^k\ell_n(t)=\frac1t$. For the
    same reason, the family $\left\{\coz(\ell_k): k\in \N\right\}$ is
    locally finite as well.
\end{proof}

The second construction is a simple direct proof of \cite[Fact
A]{MR814046} without using Stone's result in \cite[Corollary
1]{stone:48}. In fact, this construction is a modified version of
similar results in \cite{MR1476756} and \cite[Proposition
4.2]{Gutev2020}.

\begin{proposition}
  \label{proposition-Loc-Lipschitz-Ext:2}
  For each open cover $\{V_n:n\in\N\}$ of a metric space $(X,d)$ there
  exists a locally finite open cover $\{U_n:n\in \N\}$ with
  $\overline{U_n}\subset V_n$ for every $n\in \N$.
\end{proposition}

\begin{proof}
  For each $n\in \N$, take a $1$-Lipschitz function
  $\eta_n:X\to [0,2^{-n}]$ such that $V_n=\coz(\eta_n)$. Then the
  function $\eta=\sum_{n=1}^\infty\frac{\eta_n}{2^n}:X\to [0,1]$ is
  both positive-valued and ($1$-)Lipschitz. We may now use a
  construction by M. Mather, see \cite[Lemma]{MR0281155} and
  \cite[Lemma 5.1.8]{engelking:89}, in fact its further refinement for
  countable covers in \cite{MR1476756}. Namely, for each $n\in\N$,
  define a continuous function $\gamma_n:X\to [0,1]$ by
  ${\gamma_n(x)=\max\left\{\eta_n(x)-\frac12 \eta(x),0\right\}}$,
  $x\in X$, and set $U_n=\coz(\gamma_n)$. Evidently,
  $\overline{U_n}=\supp(\gamma_n)\subset \coz(\eta_n)=V_n$ because
  $\supp(\gamma_n)\subset\left\{x\in X: \eta_n(x)\geq \frac 12
    \eta(x)\right\}$. Moreover, $p\in X$ implies that $\eta(p)>2^{-k}$
  for some $k\in \N$, hence $p$ is contained in an open set
  $U\subset X$ with $\eta(x)>2^{-k}$ for every $x\in U$. Accordingly,
  $\gamma_n(x)=0$ for every $x\in U$ and $n> k$, so the family
  $U_n=\coz(\gamma_n)$, $n\in \N$, is locally finite. It is also a
  cover of $X$. Indeed, if $p\in X$, then
  $\sup_{n\in\N}\eta_n(p)\geq 2^{-k}$ for some $k\in \N$. Since
  $\eta_n(p)<2^{-k}$ for every $n> k$, it follows that
  $\sup_{n\in\N}\eta_n(p)=\eta_m(p)$ for some $m\leq k$. Thus,
  $p\in U_m=\coz(\gamma_m)$ because
  $\eta_m(p)=\sum_{n=1}^\infty \frac{\eta_m(p)}{2^n}\geq
  \eta(p)>\frac12 \eta(p)$.
\end{proof}

\begin{proof}[Proof of Theorem \ref{theorem-Loc-Lipschitz-Ext:2}]
  Having established Propositions
  \ref{proposition-Loc-Lipschitz-Ext:4} and
  \ref{proposition-Loc-Lipschitz-Ext:2}, we can proceed as in the
  proof of \cite[Lemma]{MR814046}. Namely, let
  $\ell_k:(0,+\infty)\to [0,1]$, $k\in \N$, be the sequence of
  Lipschitz functions constructed in Proposition
  \ref{proposition-Loc-Lipschitz-Ext:4}. Also, let $\{V_n:n\in \N\}$
  be an open cover of a metric space $(X,d)$, and $\{U_n:n\in \N\}$ be
  the corresponding locally finite open refinement constructed in
  Proposition \ref{proposition-Loc-Lipschitz-Ext:2}.  Next, for each
  $n\in \N$, take a $\frac1{2^n}$-Lipschitz function
  $\eta_n:X\to [0,1]$ with $U_n=\coz(\eta_n)$.  Then the sum function
  $\eta=\sum_{n=1}^\infty \eta_n:X\to [0,+\infty)$ is both
  positive-valued and ($1$-)Lipschitz. Hence, each
  $\ell_k\circ \eta:X\to [0,1]$, $k\in \N$, is a Lipschitz function
  such that by \eqref{eq:Loc-Lipschitz-Ext:2},
  $\frac 1\eta= \sum_{k=1}^\infty \ell_k\circ \eta$. Since the product
  of bounded Lipschitz functions is also Lipschitz, we can finally
  define the required partition of unity
  $\left\{\xi_{nk}:n,k\in \N\right\}$ subordinated to
  $\left\{V_n:n\in \N\right\}$ by
  $\xi_{nk}=\eta_n\cdot \left(\ell_k\circ \eta\right)$, for
  $n,k\in \N$.
\end{proof}

A partition of unity $\xi_\alpha:X\to [0,1]$, $\alpha\in \mathscr{A}$,
on a space $X$ is \emph{index-subordinated} to a cover
$\{V_\alpha:\alpha\in \mathscr{A}\}$ of $X$ if
$\supp(\xi_\alpha)\subset V_\alpha$ for every $\alpha\in
\mathscr{A}$. Regarding the difference between Lipschitz and locally
Lipschitz partitions of unity, Theorem
\ref{theorem-Loc-Lipschitz-Ext:2} implies the following result, it was
obtained in \cite[Proposition 4.2]{Gutev2020}.

\begin{corollary}
  \label{corollary-Loc-Lipschitz-Ext:1}
  Each open cover $\{V_n:n\in\N\}$ of a metric space $(X,d)$ admits a
  locally finite index-subordinated partition of unity consisting of
  locally Lipschitz functions.
\end{corollary}

\begin{proof}
  By Theorem \ref{theorem-Loc-Lipschitz-Ext:2}, the cover
  $\{V_n:n\in\N\}$ admits a locally finite partition of unity
  $\{\eta_\alpha:\alpha\in \mathscr{A}\}$ consisting of Lipschitz
  functions. Then by definition, for each $\alpha\in \mathscr{A}$
  there exists $n(\alpha)\in \N$ with
  $\supp(\eta_\alpha)\subset V_{n(\alpha)}$. For convenience, set
  $\mathscr{A}_n=\{\alpha\in \mathscr{A}: n(\alpha)=n\}$, $n\in
  \N$. Next, for each $n\in \N$ define a function $\xi_n:X\to [0,1]$
  by $\xi_n=\sum_{\alpha\in \mathscr{A}_n} \eta_\alpha$ if
  $\mathscr{A}_n\neq \emptyset$ and $\xi_n\equiv 0$ otherwise. Thus,
  we get a locally finite partition of unity $\{\xi_n:n\in \N\}$ which
  is index-subordinated to the cover $\{V_n:n\in\N\}$. Moreover, each
  $\xi_n$, $n\in\N$, is locally Lipschitz being a locally finite sum
  of Lipschitz functions.
\end{proof}

A map $f:X\to Y$ between metric spaces $(X,d)$ and $(Y,\rho)$ is
called \emph{nonexpansive} if it is $1$-Lipschitz, namely when
$\rho(f(p),f(q))\leq d(p,q)$ for every $p,q\in X$. In \cite{MR814046},
Frol\'{\i}k actually showed that each open cover of a metric space
admits a locally finite partition of unity consisting of nonexpansive
functions. Since each Lipschitz function $\xi:X\to \R$ is a finite sum
of nonexpansive functions, Theorem~\ref{theorem-Loc-Lipschitz-Ext:2}
implies also the following consequence.

\begin{corollary}
  \label{corollary-LE-Revisited-v14:1}
    Each countable open cover of a metric space admits a locally
  finite countable partition of unity consisting of nonexpansive
  functions.  
\end{corollary}

\begin{remark}
  \label{remark-LE-Revisited-v4:1}
  In several sources, including \cite{Gutev2020} and Frol\'{\i}k's
  paper \cite{MR814046}, a partition of unity $\xi_\alpha:X\to [0,1]$,
  $\alpha\in \mathscr{A}$, on space $X$ is \emph{subordinated} to a
  cover $\mathscr{V}$ of $X$ if the cover
  $\left\{\coz(\xi_\alpha):\alpha\in\mathscr{A}\right\}$ is a
  refinement of $\mathscr{V}$. According to
  Theorem~\ref{theorem-Loc-Lipschitz-Ext:2}, both interpretations are
  equivalent in the realm of countable open covers of metric
  spaces. In fact, they are also equivalent in general. Namely, if
  $\{\eta_\alpha:\alpha\in \mathscr{A}\}$ is a partition of unity on
  $X$, then $X$ also has a (locally finite) partition of unity
  $\{\xi_\alpha:\alpha\in \mathscr{A}\}$ with
  $\supp(\xi_\alpha)\subset \coz(\eta_\alpha)$ for all
  $\alpha\in \mathscr{A}$, see \cite[Proposition
  2.7.4]{MR3099433}. This property is essentially the construction of
  M. Mather used in the proof of
  Proposition~\ref{proposition-Loc-Lipschitz-Ext:2}.
\end{remark}

\section{Lipschitz and Locally Lipschitz Functions}
\label{sec:lipsch-locally-lipsc-1}

By Proposition \ref{proposition-Loc-Lipschitz-Ext:4}, the locally
Lipschitz function $\frac1t:(0,+\infty)\to (0,+\infty)$ is a locally
finite sum of bounded Lipschitz functions. Here, we will show that
this is valid for any locally Lipschitz function.

\begin{theorem}
  \label{theorem-Loc-Lipschitz-Ext:3}
  If $(X,d)$ is a metric space, then each locally Lipschitz function
  $f:X\to \R$ is a locally finite sum of some sequence of bounded
  Lipschitz functions.
\end{theorem}

The proof of this theorem is based on the following simple observation
which was obtained in \cite[Proposition 4.3]{Gutev2020}.

\begin{proposition}
  \label{proposition-Loc-Lipschitz-v5:2}
  If $(X,d)$ and $(Y,\rho)$ are metric spaces and $f:X\to Y$ is a
  bounded locally Lipschitz map, then $X$ has an open increasing cover
  $\{U_n:n\in\N\}$ such that $f\uhr U_n$ is $n$-Lipschitz for every
  $n\in \N$.
\end{proposition}

\begin{proof}
  Let $K\geq 0$ be such that $\rho(f(x),f(y))\leq K$ for every
  $x,y\in X$.  Also, for ${p\in X}$, let $\delta_p>0$ and $K_p\geq 0$
  be such that $\rho(f(x),f(y))\leq K_p d(x,y)$ for every
  $x,y\in \mathbf{O}(p,2\delta_p)$. Then just as in the proof of
  Proposition \ref{proposition-Loc-Lipschitz-v3:2},
  $L_p=\max\left\{K_p,\frac K{\delta_p}\right\}$ has the property that
  \begin{equation}
    \label{eq:Loc-Lipschitz-v6:1}
    \rho(f(x),f(y))\leq L_p d(x,y)\quad \text{for every $x\in X$
      and $y\in \mathbf{O}(p,\delta_p)$.}
  \end{equation}
  Finally, set
  $U_n=\bigcup\{\mathbf{O}(p,\delta_p): L_p\leq n\}$,
  $n\in\N$. Evidently, $\{U_n:n\in\N\}$ is an increasing open cover of
  $X$. Moreover, if $x,y\in U_n$, then $x\in \mathbf{O}(p,\delta_p)$
  and $y\in \mathbf{O}(q,\delta_q)$ for some $p,q\in X$ with
  $\max\{L_p,L_q\}\leq n$. According to
  \eqref{eq:Loc-Lipschitz-v6:1}, this implies that
  $\rho(f(x),f(y))\leq n d(x,y)$.
\end{proof}

\begin{proof}[Proof of Theorem \ref{theorem-Loc-Lipschitz-Ext:3}]
  Take a locally Lipschitz function $f:X\to \R$. Then $f$ is both
  bounded and locally Lipschitz on each set
  $f^{-1}\left((-n,n)\right)$, $n\in \N$. Hence, by Proposition
  \ref{proposition-Loc-Lipschitz-v5:2}, there exists a countable open
  cover $\mathscr{U}$ of $X$ such that $f\uhr U$ is bounded and
  Lipschitz for each $U\in \mathscr{U}$. We can now apply Theorem
  \ref{theorem-Loc-Lipschitz-Ext:2} to get a locally finite partition
  of unity $\{\xi_n:n\in \N\}$ which consists of Lipschitz functions
  and is subordinated to $\mathscr{U}$. Thus, each restriction
  $g_n=f\uhr \supp(\xi_n)$, $n\in \N$, is a bounded Lipschitz function
  and according to Proposition \ref{proposition-Lipschitz-ext:11}, it
  can be extended to a bounded Lipschitz function $\psi_n:X\to
  \R$. Finally, for each $n\in \N$, set $\varphi_n=\psi_n\cdot\xi_n$
  which is a (bounded) Lipschitz function because so are $\psi_n$ and
  $\xi_n$. To see that these functions $\varphi_n$, $n\in\N$, are as
  required, take a point $p\in X$ and consider the finite set
  $\sigma_p=\{n\in\N: \xi_n(p)>0\}$. Then
  $\varphi_n(p)=\psi_n(p)\cdot\xi_n(p)= 0$ for every
  $n\in \N\setminus \sigma_p$ and, therefore,
  \[
        f(p)= \sum_{n\in\sigma_p}
    g_n(p)\cdot \xi_n(p)
    = \sum_{n\in\sigma_p}
    \psi_n(p)\cdot \xi_n(p)=\sum_{n\in\sigma_p}\varphi_n(p)=
      \sum_{n=1}^\infty \varphi_n(p).\qedhere
    \]
\end{proof}

Precisely as in the special case of Lipschitz partitions of unity in
Theorem \ref{theorem-Loc-Lipschitz-Ext:2} (see Corollary
\ref{corollary-LE-Revisited-v14:1}), we have the following consequence
of Theorem \ref{theorem-Loc-Lipschitz-Ext:3}.

\begin{corollary}
  \label{corollary-Loc-Lipschitz-Ext:2}
  If $(X,d)$ is a metric space, then each locally Lipschitz function
  $f:X\to \R$ is a locally finite sum of some sequence of bounded
  nonexpansive functions.
\end{corollary}

\begin{remark}
  \label{remark-LE-Revisited-v11:1}
  It is a simple exercise that the product of two locally pointwise
  Lipschitz functions is also locally pointwise Lipschitz. Similarly,
  a locally finite sum of locally pointwise Lipschitz functions is
  locally pointwise Lipschitz as well. Accordingly, Theorem
  \ref{theorem-Loc-Lipschitz-Ext:3} remains valid for locally
  pointwise Lipschitz functions with essentially the same proof but
  now using Theorem \ref{theorem-Loc-Lipschitz-v17:1} instead of
  Proposition~\ref{proposition-Lipschitz-ext:11}. Namely, a function
  $f:X\to \R$ is locally pointwise Lipschitz of and only if it is a
  locally finite sum of (bounded) pointwise Lipschitz functions.
\end{remark}

We conclude this section with another result related to the difference
between Lipschitz and locally Lipschitz functions. Namely, using
Proposition \ref{proposition-Loc-Lipschitz-v5:2} we will also prove
the following theorem.

\begin{theorem}
  \label{theorem-LE-Revisited-v17:1}
  For metric spaces $(X,d)$ and $(Y,\rho)$, a map $f:X\to Y$ is
  locally Lipschitz if and only if there exists a continuous function
  $L:X\times X\to [0,+\infty)$ such that
  \begin{equation}
    \label{eq:LE-Revisited-v17:1}
    \rho(f(x),f(y))\leq L(x,y) d(x,y)\quad \text{for every $x,y\in X$.}
  \end{equation}
\end{theorem}

To prepare for the proof of Theorem \ref{theorem-LE-Revisited-v17:1},
let us recall that a function $\eta:X\to \R$ is \emph{upper}
(\emph{lower}) \emph{semi-continuous} if the set
\[
\{x\in X:\eta(x)<t\}\quad \text{(respectively, $\{x\in X:\eta(x)>t\}$)}
\]
is open in $X$ for every $t\in \R$. In \cite{hausdorff:19}, Hausdorff
showed that Baire's approximation theorem for semi-continuous in
\cite{MR1504475} is valid for any metric domain $(X,d)$, namely that
each upper (lower) semi-continuous function $\eta:X\to \R$ is the
pointwise limit of some decreasing (increasing) sequence of continuous
functions. A partial case of this result was previously obtained by
Hahn in \cite{zbMATH02610879}, and refined in his book
\cite{Hahn1921}. The interested reader is also referred to
\cite[Theorem 5.1]{Gutev2020} for a simple proof that in this case,
each upper (lower) semi-continuous function $\eta:X\to \R$ is actually
the pointwise limit of some decreasing (increasing) sequence of
locally Lipschitz functions. \label{page:lipsch-locally-lipsc-2} Based
on this, we have the following special case of Theorem
\ref{theorem-LE-Revisited-v17:1}.

\begin{proposition}
  \label{proposition-LE-Revisited-v17:1}
  Let $(X,d)$ and $(Y,\rho)$ be metric spaces. Then each bounded
  locally Lipschitz map $f:X\to Y$ satisfies
  \eqref{eq:LE-Revisited-v17:1} with respect to some continuous
  function $L:X\times X\to [0,+\infty)$.
\end{proposition}

\begin{proof}
  Let $f:X\to Y$ be a bounded locally Lipschitz map. Then by
  Proposition~\ref{proposition-Loc-Lipschitz-v5:2}, $X$ has an
  increasing open cover $\{U_n:n\in\N\}$ such that $f\uhr U_n$ is
  $n$-Lipschitz for each $n\in \N$. Next, define a function
  $\eta:X\to \N\subset \R$ by
  \[
    \eta(x)=\min\{n\in\N: x\in U_n\}\quad \text{for every $x\in X$.}
  \]
  Since the cover $U_n$, $n\in \N$, is increasing and open, $\eta$ is
  upper semi-continuous. Indeed, if $x\in X$ and $t\in\R$ with
  $\eta(x)=n<t$, then $x\in U_n$ and ${\eta(y)\leq n<t}$ for every
  $y\in U_n$. Thus, we can now use the aforementioned Hausdorff's
  result in \cite{hausdorff:19}. Hence, in particular, there exists a
  continuous function $\ell:X\to \R$ such that $\eta\leq
  \ell$. Finally, we can define the required function
  $L:X\times X\to [0,+\infty)$ by ${L(x,y)=\max\{\ell(x),\ell(y)\}}$,
  $x,y\in X$. To see that $L(x,y)$ is as in
  \eqref{eq:LE-Revisited-v17:1}, take $x,y\in X$ and for convenience,
  set $n=\max\{\eta(x),\eta(y)\}$. Then $x,y\in U_n$ and by condition,
  $f\uhr U_n$ is $n$-Lipschitz. Accordingly,
  \begin{align*}
    \rho(f(x),f(y))\leq n d(x,y)
    &= \max\{\eta(x),\eta(y)\} d(x,y)\\
    &\leq
      \max\{\ell(x),\ell(y)\} d(x,y)=L(x,y) d(x,y).\qedhere 
  \end{align*}
\end{proof}

\begin{proof}[Proof of Theorem \ref{theorem-LE-Revisited-v17:1}]
  Let $(X,d)$ and $(Y,\rho)$ be metric spaces, and $f:X\to Y$. If
  $L:X\times X\to [0,+\infty)$ is a continuous function as in
  \eqref{eq:LE-Revisited-v17:1}, then it is locally
  bounded. Accordingly, $f$ satisfies the Lipschitz condition locally
  and is, therefore, locally Lipschitz. Conversely, assume that
  $f:X\to Y$ is locally Lipschitz, and consider the equivalent bounded
  metric $\rho_*$ on $Y$ defined by
  $\rho_*(y,z)=\frac{\rho(y,z)}{1+\rho(y,z)}$ for $y,z\in Y$. Then $f$
  remains locally Lipschitz as a map from $(X,d)$ to $(X,\rho_*)$
  because $\rho_*\leq \rho$. Hence, by Proposition
  \ref{proposition-LE-Revisited-v17:1}, there exists a continuous
  function $L_*:X\times X\to [0,+\infty)$ satisfying
  \eqref{eq:LE-Revisited-v17:1} with respect to the metric $\rho_*$,
  i.e.\ for which $\rho_*(f(x),f(y))\leq L_*(x,y) d(x,y)$ for every
  $x,y\in X$. Evidently, the function
  $L(x,y)= \left[1+\rho(f(x),f(y)\right] L_*(x,y)$, for $x,y\in X$, is
  now as in \eqref{eq:LE-Revisited-v17:1} with respect to the metric
  $\rho$ on $Y$.
\end{proof}

\section{Locally Lipschitz Extensions and Selections}
\label{sec:extens-locally-lipsc}

Unlike (pointwise) Lipschitz functions, the extension of locally
Lipschitz functions is based on that of Lipschitz functions and
Lipschitz partitions of unity. Thus, in \cite[Theorem 4.1]{Gutev2020}
(see also \cite[Theorem 5.12]{MR515647}) it was shown that for a
closed subset $A\subset X$ of a metric space $(X,d)$, each locally
Lipschitz function $\varphi:A\to \R$ can be extended to a locally
Lipschitz function $f:X\to \R$. This result was also stated in
\cite[Theorem 4.1.7]{zbMATH07045619} and credited to Czipszer and
Geh\'er \cite{MR71493}, but locally Lipschitz extensions were not
discussed \cite{MR71493} and the proof given in \cite{zbMATH07045619}
is for locally pointwise Lipschitz functions, see Section
\ref{sec:extens-pointw-lipsch}. In fact, the author is not aware of
any explicit formula that will result in locally Lipschitz
extensions.\medskip

Here, we will obtain the following more general result. 

\begin{theorem}
  \label{theorem-Loc-Lipschitz-v5:1}
  Let $(X,d)$ be a metric space, $A\subset X$ be a closed set and
  $\Delta\subset \R$ be an interval. Then each locally Lipschitz
  function $\varphi:A\to \Delta$ can be extended to a locally
  Lipschitz function $f:X\to \Delta$.
\end{theorem}

\begin{proof}
  As in the proof of Theorem \ref{theorem-Loc-Lipschitz-Ext:3}, $A$
  has an open cover $\{U_n:n\in\N\}$ such that each restriction
  $\varphi\uhr U_n$, $n\in\N$, is Lipschitz. Since $A$ is closed, $X$
  has an open cover $\{V_n:n\in\N\}$ with $U_n=V_n\cap A$,
  $n\in\N$. Then by Corollary~\ref{corollary-Loc-Lipschitz-Ext:1},
  $\{V_n:n\in\N\}$ has an index-subordinated locally finite partition
  of unity ${\{\xi_n:n\in\N\}}$ consisting of locally Lipschitz
  functions. Now, for each $n\in \N$, set
  $A_n=A\cap \supp(\xi_n)\subset U_n$ and $\varphi_n=\varphi\uhr
  A_n$. Evidently, $\varphi_n$ is Lipschitz because so is
  $\varphi\uhr U_n$. Hence, according to
  Theorem~\ref{theorem-Lipschitz-ext:20}, each
  $\varphi_n:A_n\to \Delta$ can be extended to a Lipschitz function
  $f_n:X\to \Delta$. Finally, we can define a function $f:X\to \R$ by
  $f=\sum_{n=1}^\infty \xi_n\cdot f_n$. Since the partition of unity
  $\{\xi_n:n\in\N\}$ is locally finite and each product function
  $\xi_n\cdot f_n$, $n\in \N$, is locally Lipschitz, the function $f$
  is also locally Lipschitz. To see that this $f$ is as required, take
  a point $p\in X$, and consider the finite set
  $\sigma_p=\{n\in\N: \xi_n(p)>0\}$. If $p\in A$, then
  $f_n(p)=\varphi_n(p)=\varphi(p)$ for every $n\in\sigma_p$, and
  therefore
  $f(p)=\sum_{n\in\sigma_p} \xi_n(p)\cdot f_n(p) =\varphi(p)\cdot
  \sum_{n\in\sigma_p} \xi_n(p)= \varphi(p)$.  If $p\in X\setminus A$,
  then $f_n(p)\in \Delta$ for every $n\in \N$. Hence, for the same
  reason,
  \[
    f(p)=\sum_{n\in\sigma_p} \xi_n(p)\cdot f_n(p)\in \Delta.\qedhere
  \]
\end{proof}

According to Remark \ref{remark-LE-Revisited-v11:1}, the proof of
Theorem \ref{theorem-Loc-Lipschitz-v5:1} gives an alternative approach
to the last part of the proof of Theorem
\ref{theorem-Lipschitz-ext:11}, i.e.\ the same proof works to show
Theorem \ref{theorem-Lipschitz-ext:11} for the case of an unbounded
interval $\Delta\subset \R$.\medskip

We conclude this paper with a result about locally Lipschitz
selections of set-valued mappings. To this end, let us recall some
terminology. For spaces $X$ and $Y$, we write $\Omega:X\sto Y$ to
designate that $\Omega$ is a map from $X$ to the \emph{nonempty}
subsets of $Y$. In Michael's selection theory, such a map is commonly
called a \emph{set-valued mapping} (also a \emph{multifunction}, or
simply a \emph{carrier} \cite{michael:56a}). A usual map $f:X\to Y$ is
a \emph{selection} for $\Omega:X\sto Y$ if $f(x)\in\Omega(x)$ for
every $x\in X$. Finally, let us recall that a mapping $\Omega:X\sto Y$
is \emph{lower semi-continuous}, or \emph{l.s.c.}, if the preimage
$\Omega^{-1}[U]=\{x\in X: \Omega(x)\cap U\neq \emptyset\}$ is open in
$X$ for every open $U\subset Y$. \medskip

In \cite[Theorem 3.1$'''$]{michael:56a}, Michael showed that for a
perfectly normal space $X$, each convex-valued l.s.c.\ mapping
$\Omega:X\sto \R$ has a continuous selection. Furthermore, it follows
from \cite[Example 1.3$^*$]{michael:56a} that in this case, for each
closed $A\subset X$, each continuous selection $\varphi:A\to \R$ for
$\Omega\uhr A$ can be extended to a continuous selection for
$\Omega$. For an extended discussion on this result and related
selection results, the interested reader is referred to
\cite{Gutev2023}. For instance, it follows from \cite[Corollary
3.5]{Gutev2023} and the proof of \cite[Corollary 4.5]{Gutev2023} that
for an open-convex-valued l.s.c.\ mapping $\Omega:X\sto \R$ on a
countably paracompact normal space $X$ and a closed subset
$A\subset X$, each continuous selection for $\Omega\uhr A$ can be
extended to a continuous selection for $\Omega$. Here, we will refine
this result in the setting of a metric space $(X,d)$. Namely, the
following theorem will be proved.

\begin{theorem}
  \label{theorem-LE-Revisited-v14:1}
  Let $(X,d)$ be a metric space and $\Omega:X\sto \R$ be an
  open-convex-valued l.s.c.\ mapping. Then for each closed subset
  $A\subset X$, each locally Lipschitz selection for $\Omega\uhr A$
  can be extended to a locally Lipschitz selection for $\Omega$.
\end{theorem}

The proof of Theorem \ref{theorem-LE-Revisited-v14:1} is based on two
observations, the first of which is essentially the proof of
\cite[Corollary 4.5]{Gutev2023} and is valid for any space $X$.

\begin{proposition}
  \label{proposition-LE-Revisited-v14:1}
  Each open-convex-valued l.s.c.\ mapping $\Omega:X\sto \R$ has an
  open graph $\Gamma(\Omega)=\{(x,t)\in X\times \R: t\in \Omega(x)\}$.
\end{proposition}

\begin{proof}
  As in the proof of \cite[Corollary 4.5]{Gutev2023}, take $p\in X$
  and $s,t\in \Omega(p)$ with $s<t$. Then
  $(-\infty,s)\cap \Omega(p)\neq \emptyset$ because $\Omega(p)$ is
  open. Similarly, $(t,+\infty)\cap \Omega(p)\neq \emptyset$. Hence,
  $V=\Omega^{-1}\left[(-\infty,s)\right]\cap
  \Omega^{-1}\left[(t,+\infty)\right]$ is an open set with $p\in
  V$. Moreover, $V\times[s,t]\subset \Gamma(\Omega)$ because $\Omega$
  is convex-valued.
\end{proof}

The second observation is the special case of Theorem
\ref{theorem-LE-Revisited-v14:1} when $A=\emptyset$.

\begin{theorem}
  \label{theorem-Lipschitz-vgg-r1:2}
  If $(X,d)$ is a metric space, then each open-convex-valued l.s.c.\
  mapping $\Omega:X\sto \R$ has a locally Lipschitz selection.
\end{theorem}

\begin{proof}
  We follow the construction in \cite[Remark 4.7]{Gutev2020}, see also
  the proof of \cite[Theorem 1.1]{Gutev2023}. Namely, let
  $\Omega:X\sto \R$ be an open-convex-valued l.s.c.\ mapping.  Then by
  Proposition~\ref{proposition-LE-Revisited-v14:1}, each set
  $U_r=\Omega^{-1}[\{r\}]$, $r\in\Q$, is open in $X$. Moreover,
  $\{U_r: r\in\Q\}$ is a cover of $X$ because $\Omega$ is
  open-valued. Since the rational numbers $\Q$ are countable, by
  Corollary~\ref{corollary-Loc-Lipschitz-Ext:1}, $\{U_r: r\in\Q\}$ has
  an index-subordinated locally finite partition of unity
  $\{\xi_r:r\in\Q\}$ consisting of locally Lipschitz
  functions. Finally, using that $\Omega$ is also convex-valued, the
  required locally Lipschitz selection $f:X\to \R$ is defined by
  $f(x)=\sum_{r\in\Q}\xi_r(x)\cdot r$, $x\in X$.
\end{proof}

\begin{proof}[Proof of Theorem \ref{theorem-LE-Revisited-v14:1}]
  Let $A\subset X$ be a nonempty closed subset and ${\varphi:A\to \R}$
  be a locally Lipschitz selection for $\Omega\uhr A$. We can follow
  the proof of \cite[Corollary 3.5]{Gutev2023}. Briefly, by Theorem
  \ref{theorem-Loc-Lipschitz-v5:1}, $\varphi$ can be extended to a
  locally Lipschitz function ${g:X\to \R}$. Then by Proposition
  \ref{proposition-LE-Revisited-v14:1},
  $U=\{x\in X: g(x)\in \Omega(x)\}$ is an open set with $A\subset
  U$. The essential case in this proof is when
  ${B=X\setminus U\neq \emptyset}$. In this case, take a locally
  Lipschitz function $\xi:X\to [0,1]$ such that $A=\xi^{-1}(0)$ and
  ${B= \xi^{-1}(1)}$, for instance
  $\xi(x)=\frac{d(x,A)}{d(x,A)+d(x,B)}$ for $x\in X$. Next, using
  Theorem \ref{theorem-Lipschitz-vgg-r1:2}, take a locally
  Lipschitz selection $h:X\to \R$ for $\Omega$. Finally, define
  another locally Lipschitz function $f:X\to \R$ by
  \[
    f(x)=(1-\xi(x))\cdot g(x)+\xi(x)\cdot h(x)\quad \text{for every
      $x\in X$.}
  \]
  Evidently, $f$ is a locally Lipschitz selection for $\Omega$ with
  $f\uhr A=\varphi$.
\end{proof}

Special cases of Theorem \ref{theorem-LE-Revisited-v14:1} are known in
different terms. Namely, Dowker showed in \cite[Theorem 4]{dowker:51}
that a space $X$ is countably paracompact and normal if and only if
for each pair of functions $g,h:X\to \R$ such that $g$ is upper
semi-continuous, $h$ is lower semi-continuous and $g<h$, there exists
a continuous function $f:X\to \R$ with $g <f<h$. In this case, as
shown in the proof of \cite[Proposition 4.3]{Gutev2023},
$\Omega(x)=(g(x), h(x))$, $x\in X$, defines an open-convex-valued
l.s.c.\ mapping $\Omega:X\sto \R$. Accordingly, for a metric space
$(X,d)$, Theorem \ref{theorem-LE-Revisited-v14:1} contains the
following refinement of the aforementioned Dowker's result.

\begin{corollary}
  \label{corollary-LE-Revisited-v18:1}
  Let $(X,d)$ be a metric space and $g,h:X\to \R$ be functions such
  that $g$ is upper semi-continuous, $h$ is lower semi-continuous and
  $g<h$. Also, let $A\subset X$ be a closed set and $\varphi:A\to \R$
  be a locally Lipschitz function with $g\uhr A< \varphi<h\uhr
  A$. Then $\varphi$ can be extended to a locally Lipschitz function
  $f:X\to \R$ such that $g< f< h$.
\end{corollary}

The special case of $A=\emptyset$ in Corollary
\ref{corollary-LE-Revisited-v18:1} was obtained in \cite[Theorem
5.4]{MR515647}, see also \cite[Remark 4.7]{Gutev2020}. The argument
suggested in \cite{MR515647} is to use paracompactness of metrizable
spaces and follow an alternative proof of Dowker's insertion theorem
given in \cite[4.3, p.\ 171]{dugundji:66}, but now to take a locally
Lipschitz partition of unity. Here, we used the same approach but
based entirely on Corollary \ref{corollary-Loc-Lipschitz-Ext:1}. Let
us remark that Theorem \ref{theorem-LE-Revisited-v14:1} also contains
the case when the functions $g$ and $h$ in Corollary
\ref{corollary-LE-Revisited-v18:1} may take infinite values. In this
case, just like before, the mapping
$X\ni x\sto \Omega(x)=(g(x),h(x))\subset \R$ is l.s.c.\ and
open-convex-valued.\medskip

As remarked before (see page \pageref{page:lipsch-locally-lipsc-2}),
it was shown in \cite[Theorem 5.1]{Gutev2020} that for a metric space
$(X,d)$, each upper (lower) semi-continuous function $\eta:X\to \R$ is
the pointwise limit of some decreasing (increasing) sequence of
locally Lipschitz functions. Theorem \ref{theorem-Lipschitz-vgg-r1:2}
implies the following complementary result.

\begin{corollary}
  \label{corollary-Lipschitz-vgg-r1:1}
  For a metric space $(X,d)$, each continuous function $\varphi:X\to
  \R$ is the uniform limit of some strictly decreasing sequence of
  locally Lipschitz functions. 
\end{corollary}

\begin{proof}
  It suffices to construct a sequence $f_n:X\to \R$, $n\in \N$, of
  locally Lipschitz functions such that $\varphi<f_1<\varphi+1$ and
  $\varphi< f_{n+1}<\frac{\varphi+f_n}2$ for every $n\in \N$. To this
  end, according to Theorem \ref{theorem-Lipschitz-vgg-r1:2}, there
  exists a locally Lipschitz function $f_1:X\to \R$ with
  $\varphi<f_1<\varphi+1$, because mapping
  $X\ni x\sto (\varphi(x), \varphi(x)+1)\subset \R$ is l.s.c.\ and
  open-convex-valued. Similarly, $\varphi<f_2<\frac{\varphi+ f_1}2$
  for some locally Lipschitz function $f_2:X\to \R$. The construction
  can be carried on by induction.
\end{proof}

It is well known that each continuous function on a metric space is a
uniform limit of locally Lipschitz functions, see for instance
Miculescu \cite[Theorem 2]{MR1825525} and Garrido and Jaramillo
\cite[Corollary 2.8]{MR2037989}. A simple and self-contained proof of
this fact was also given in \cite[Proposition 5.3]{Gutev2020}.\medskip

Regarding uniform approximations by locally Lipschitz functions, let
us recall that for metric spaces $(X,d)$ and $(Y,\rho)$, a map
$f:X\to Y$ is called \emph{Lipschitz in the small} \cite{MR538094} if
there exists $\delta>0$ and $K>0$ such that each restriction
$f\uhr \mathbf{O}(x,\delta)$ is $K$-Lipschitz, for every $x\in
X$. Evidently, each Lipschitz in the small map is locally Lipschitz,
but the converse is not true. For instance, the locally Lipschitz
function $t^2:\R\to \R$ is not Lipschitz in the small. Also, there are
simple examples of functions that are Lipschitz in the small, but are
not Lipschitz. For a metric space $(X,d)$, each uniformly continuous
function $\varphi:X\to \R$ is a uniform limit of functions which are
Lipschitz in the small. For a bounded uniformly continuous function
$\varphi:X\to \R$, this result can be found in \cite[Theorem
6.8]{MR1800917}, see also \cite[Theorem 1]{MR1973966}. For an
arbitrary uniformly continuous function, the result was obtained by
Garrido and Jaramillo \cite[Theorem 1]{MR2376153}. Subsequently, a
simple proof of this result was given by Beer and Garrido in
\cite[Theorem 6.1]{MR3334948}, and the proof was further simplified in
\cite[Theorem 5.6]{Gutev2020}.

\subsection*{Acknowledgement.} In conclusion, the author would like to 
express his sincere gratitude to the referee for several valuable
remarks and suggestions.

\end{document}